\documentclass[12pt,a4paper]{amsart}
\usepackage{amssymb,amsmath,a4wide,babel}

\newtheorem{theorem}{Theorem}[section]

\newtheorem*{corollary}{Corollary}

\newtheorem*{remark}{\it Remarks\/}
\newtheorem{rem}{\it Remark\/}
\newtheorem*{remk}{\it Remark\/}

\newcommand{\la}{\lambda}
\newcommand{\RR}{{\mathbb{R}}}
\renewcommand{\div}{\mathop{{\rm div}}}
\newcommand{\tq}{\;\vrule height 12pt depth 5pt\;}
\newcommand{\meas}{\mathop{{\rm meas}}}
\newcommand{\vol}{\mathop{{\rm vol}}}

\title[Symmetrization on the Sphere]{Symmetrization on the sphere and applications}%
\author{Satyanad Kichenassamy}%
\address{D. M. I., \'Ecole Normale Sup\'erieure, 45 rue d'Ulm,
75230 Paris Cedex 05, France}%
\address{Courant Institute of Mathematical Sciences, 251 Mercer Street, New
York, N.Y. 10012}
\thanks{Appeared in: \emph{Nonlinear Partial Differential Equations and Applications}, Coll\`ege de France Seminar (1987-1988), vol.~X,
H.~Brezis and J.-L.~Lions (eds.), Pitman Research Notes in Mathematics, \textbf{220}, (1991) 271-283.}%
\begin{document}
\maketitle

\begin{abstract} [added 2025] We introduce a new method of symmetrization of mappings on the $n$-sphere ($n\geq 2$). They are applied to estimate solutions of quasilinear elliptic partial differential equations of $p$-Laplacian type, with combinations of Dirac measures on the right-hand side. The case $p=n$ is reduced to a problem on the sphere, using a conformal transformation. The cases when $1<p<n$ and $p>n$ are considered more briefly, full details being available in other papers of the author.
\end{abstract}

In this paper, some new results are proved on the symmetrization of
mappings on the $n$-sphere ($n\geq 2$) with weak regularity assumptions and used to estimate solutions of nonlinear elliptic partial differential
equations.

 Recall that symmetrization techniques on an open subset of Euclidean
space give estimates on the distribution function $t\mapsto \meas(u > t)$ of a
function $u$ solving a partial differential equation. The result is
interpreted as an estimate of the \emph{spherical rearrangement} (or Schwarz
symmetrization) $u^*$ of $u$ in terms of the solution of a "symmetrized"
problem.

These techniques brought about the manipulation of sets with nonsmooth
boundary, the level sets of the solution; and the relevant way of
measuring the boundary of such sets is given by de Giorgi's perimeter.

We shall here give some generalizations of the aforementioned techniques
to the sphere $S^n$, $n\geq 2$. In section 1, we define the symmetrization and rearrangement of functions defined on the sphere. We then prove (improving
Sperner [11]) that the symmetrization decreases total variation. De Giorgi's
perimeter has to be replaced by the perimeter on a manifold (following
Kichenassamy [5, 6]) which we recall here. Section 2 is concerned with the
estimation of u, weak solution of a quasilinear equation in divergence
form on $S^n$, with right-hand side in $L^1$, by an adaption of symmetrization methods on a bounded open set for symmetrization in n-space, see [1, 2, 7, 3, 9, 12, 13, 14]); these results are then used to solve a quasilinear problem in $\RR^n$, the solution of which could not be found by standard symmetrization methods.

\section{Symmetrization on the sphere}
\subsection{Definitions}

We consider throughout this section $u\in L^1(S^n)$, where the $n$-sphere $S^n$ ($n\geq 2$) [p.~273] is endowed with the Riemannian structure induced by the imbedding $S^n\subset\RR^{n+1}$ (Euclidean space). We define the distribution function $\mu$ of $u$
which associates to $t\in\RR$
\begin{equation}\label{1}
\mu(t) = \meas (u > t)
\end{equation}
and the one-dimensional decreasing rearrangement $\bar u$ which is defined on $[0,\vol(S^n)]$ by
\begin{equation}\label{2}
\bar u(s) = \inf\{t\in\RR\tq \mu(t) < s\}.
\end{equation}
The symmetrization $u^*$ of $u$ is the only mapping which has geodesic balls about the North pole as level sets and which is equimeasurable with $u$
(i.e., $\meas(u^* > t) = \meas(u > t)$ for every $t$). We define $E^*$ by $\chi_E^* = (\chi_E)^*$ for any $E\subset S^n$, measurable.

\begin{remark} \emph{ (i) Knowledge of $\mu$ is equivalent to that of $\tilde u$ or $u^*$.}

\emph{(ii) Writing $b(t,x) =\chi_{u>t\geq 0} - \chi_{u(x)\leq t<0}$ for $t\in\RR$, $x\in S^n$, we have,  for every $x$,
\begin{equation}\label{3}
u(x) = \int_{-\infty}^{+\infty}b(t,x) dt
\end{equation}
and }
\begin{equation}\label{4}
u^*(x) = \int_{-\infty}^{+\infty}b(t,x)^* dt
\end{equation}

\emph{(iii) $(u + C)^* = u^* + C$ for every $C$ real.}

\emph{(iv) $(\la u)^* = \la u^*$ for every positive $\la$.}

\emph{(v) $(-\chi_E)^*=-\chi_E^*$ a.e.}

\emph{(vi)} $u\leq v \Rightarrow u^* \leq v^*$.

\emph{(vii)}  $\|u^*-v^*\|_{L^1}\Rightarrow \|u - v\|_{L^1}$.
\end{remark}
[p.~274]
We give the proof of (vii), the other properties being simpler to prove.

\underline{Proof of (vii).} By definition of the symmetrization, one has
$\int_{S^n}\varphi(u)=\int_{S^n}\varphi(u^*)$ for $\varphi=\chi_I$, where I is an interval in $\RR$. By
approximation, the same property holds for continuous $\varphi$ (say) and in
particular
\[\int_{S^n}u=\int_{S^n}u^*.
\]
We recall an argument from [8] using a technique of Crandall and Tartar.
Let $\sup(u,v) = w$. Then, as $w^* \geq sup(u^*,v^*)$,
\begin{eqnarray*}
\int_{S^n}|u^* - v^*| &\leq &  \int_{S^n}(|u^* - w^*|+|w^* - v^*|)\\
& & \int_{S^n}(2w^*-u^*-v^*) = \int_{S^n}(2w-u-v),
\end{eqnarray*}
because symmetrization preserves the integral. Since $2w-u -v = |u-v|$,
we obtain the desired inequality.

\subsection{Perimeter on $S^n$}
We recall here (from Kichenassamy [5,6], generalizing de Giorgi [3]) the
definition and properties of the perimeter of a measurable set of a
Riemannian, compact, orientable manifold, here $S^n$. We denote by $dV$ the volume element on our manifold and define the perimeter of any measurable
subset E of that manifold by
\begin{equation}\label{5}
P(E) = \sup_{\varphi\text{ smooth vector field}, |\varphi|\leq 1}\int_E \div\varphi\; dV.
\end{equation}
This amounts to saying that $P(E)$ is the total variation of $\chi_E$ where we
define the total variation of $u\in L^1$ by [p.~275]
\begin{equation}\label{6}
V(u) = \sup_{\varphi\text{ smooth}, |\varphi|\leq 1}\int_{S^n} u\div\varphi\; dV.
\end{equation}

(i) \underline{Elementary properties.}
First, $P(S^n \setminus E) = P(E)$. Second, if $u_m \to u$ in $L^1(S^n)$, then $V(u) \leq \liminf V(u_m)$.

(ii) \underline{Smoothing property.} If $u(t)$ is the solution of the heat equation
with u as initial data one has
\begin{theorem}
    There is a constant $c\geq 0$ such that $e^{-ct}\int_{S^n} |du(t)|$ is non-increasing on ${t > a}$. Its limit as $t\downarrow 0$ is equal to $V(u)$.
\end{theorem}
This means that we can characterize functions with measures as
derivatives by the behaviour of their smoothing by the heat kernel.

(iii) \underline{Generalization of the Fleming-Rishel formula.}
\begin{theorem}
    $V(u)={}_*\!\int_{-\infty}^{+\infty}P(u > t) dt$.
\end{theorem}
This theorem generalizes the corresponding result in $n$-space due to
Fleming and Rishel [4].
\begin{proof}(a) For $\varphi$ smooth vector field of length 1 or less, one writes
\begin{equation}\label{7}
\int u\div\varphi\; dV = \int_{-\infty}^{+\infty} \left(\int b(t,x) \div\varphi\; dV(x)\right)dt,
\end{equation}
which one estimates after inversion of the integrations. As the variation
of $b(t,.)$ is $P(u > t)$ one has
\begin{equation}\label{8}
V(u) \leq{\vphantom{\int}}_*\!\int_{-\infty}^{+\infty} P(u > t)\, dt.
\end{equation}

(b) To complete the proof, one uses the smoothing property: there are
functions $u_m$ smooth enough for us to write
[p.~276]
\begin{equation*}
V(u_m) = \int |du_m| =\int_{-\infty}^{+\infty}P(u_m > t)\, dt
\end{equation*}
by the classical co-area formula, and which satisfy as $m\to\infty$,
\begin{equation}\label{9}
\|u_m - u\|_{L^1}\to 0; \quad u_m\to u\text{ a.e.}; \quad
V(u_m)\to V(u).
\end{equation}
Now for every $t$, $\chi_{u_m >t}\to \chi_{u >t}$
in $L^1$, by dominated convergence.
Therefore
\begin{equation}\label{10}
\liminf P(u_m > t) \geq P(u > t)
\end{equation}
and we have
\begin{eqnarray}
V(u) = \lim V(u_m) &=& \lim \int_{-\infty}^{+\infty}P(u_m > t) dt \nonumber \\
&\geq & \liminf \int_{-\infty}^{+\infty}P(u_m > t)\, dt \geq \int_{-\infty}^{+\infty}P(u > t)\, dt,
\end{eqnarray}
which ends the proof.

Note that this argument does not make use of any
approximation by piecewise linear functions, as was the case in [4].
\end{proof}

(iv) One also has a \underline{general isoperimetric inequality} of the form
\begin{equation}\label{12}
P(E)\geq C\min(|E|,|S^n\setminus E|)^{1-1/n}.
\end{equation}
In fact, we shall obtain below a more precise result.

\subsection{Symmetrization and total variation}
We prove here
\begin{theorem}
 $V(u^*) \leq V(u)$.
\end{theorem}
[p.~277]
Let us note a direct consequence of this property:
\begin{corollary}$P(E) \geq P(E^*)$.
\end{corollary}
\begin{proof} We may assume $V(u) <\infty$. We
then the assertion holds (Sperner [11], using an isoperimetric inequality from
from [10]). Smoothing by the heat kernel gives rise to $u_m$ such that
$u_m\to u$ in $L^1$ and a.e.
It follows that $(u_m)^* \to u^*$ in $L^1$ too. Therefore
\begin{equation}\label{13}
V(u) \geq\liminf V(u_m^*) \geq V(u^*).
\end{equation}
\end{proof}

\section{Application to nonlinear partial differential equations}
\smallskip

\textbf{2.2.} We show here that if $u\in W^{1,n}(S^n)$ satisfies
\begin{equation}\label{14}
-\div ( |Du|^{n-2} Du) = f \text{ in }S^n
\end{equation}
then for every $q > 1$ one can bound $\|u\|_{L^q(S^n)}$ in terms of $\|f\|_{L^1(S^n)}$.
This in turn will give information on solutions of
\begin{equation}\label{15}
-\div ( |Du|^{n-2} Du) = g \text{ in }\RR^n
\end{equation}
where $g$ is, say, compactly supported.
Indeed, the stereographic projection $S^n\stackrel{\pi}{\to}\RR^n$ (from the North pole of $S^n$) being conformal, if $u$ solves (15), then $u\circ\pi$ solves (14) with $f = \la g\circ\pi$ where $\la$ is a smooth function on $S^n\setminus\{\text{North}\}$. We therefore can estimate $u$ in $L^q_{\text{loc}}$ provided we know how to estimate $u\circ\pi$. This is what we achieve below.
\begin{rem} \emph{Note that if we were to try symmetrization on (15), on the $n$-ball of radius $R$ with the idea of letting $R\to\infty$, we could only obtain an estimate
of the Schwarz symmetrization of $u$ by the solution of
[p.~278]
\[-\div ( |Du_R|^{n-2} Du_R) = g
\]
with Dirichlet conditions on $B_R$. Now if $g$ is, say, $\chi_{B(1)}$, then $u_R$ tends to infinity pointwise, as is easily seen. This is related to the fact that $-\div ( |Du|^{n-2} Du) = 0$ admits as solution $C \ln(1/|x|)$ (with $C = C(n)$), which does not tend to zero as $|x|\to\infty$.}
\end{rem}
\begin{rem} \emph{These techniques work equally well with $S^n$ replaced by a bounded domain, with Neumann boundary conditions. The relevant isoperimetric inequality being that for the ordinary perimeter [4]. For another approach
to estimates in Neumann problems in Euclidean space via symmetrization,
see Maderna and Salsa [7] and Tricarico [14].}
\end{rem}
\begin{rem}
\emph{Similar estimates on more general equations may be obtained in
the spirit of Talenti's results in $n$-space [12, 13].}
\end{rem}
\smallskip

\textbf{2.2.}  We now come back to the problem of estimating $u$ such that
\[-\div ( |Du|^{n-2} Du) = f \text{ in }S^n, \quad u\in W^{1,n}(S^n).\]
We shall estimate the distribution function $\mu(t)=|u > t|$ of $u$.
We need five steps:

(a) We normalize $u$ by addition of a constant in such a way that
\[\text{$|u > 0|$ and $|u < 0|$ are both $\leq (1/2)|S^n|$.}
\]

(b) We multiply the equation by $(u - t)_+$ for $t\in\RR$, integrate and
differentiate with respect to $t$. This yields
\[-\frac d{dt}\int_{u>t}|Du|^n=\int_{u>t}f\leq \|f\|_{L^1(S^n)}.
\]

[p.~279]

(c)
\[\left(\frac{\frac d{dt}\int_{u>t}|Du|^n}{\frac d{dt}\int_{u>t}|Du|}\right)^{-1/(n-1)}
\leq\frac{C\mu'(t)}{\frac d{dt}\int_{u>t}|Du|}
\]
for almost every $t$. This can be proved by replacing $d/dt$ by a finite
difference and using Jensen's inequality for the convex function
$t\mapsto t^{-1/(n-1)}$ as in Talenti [13].

(d) As a consequence of Theorem 2, for almost every $t$,
\[-\frac d{dt}\int_{u>t}|Du|= P(u > t).
\]

(e) By the isoperimetric inequality and (a), if $t > 0$, one has
\[ P(u> t) \geq C \mu^{1-1/n}(t).
\]

Putting (b)-(e) together, we obtain
\begin{equation}\label{16}
-\mu'\geq C\mu.
\end{equation}
Now, this proves that $\overline{u^+}(s) = O(|\ln s|)$ for $s$ small, and as
\begin{equation}\label{17}
\int_{u>0}u^q=\int_0^{|S^n|}(\overline{u^+})^q ,
\end{equation}
we obtain an estimate of $\|u^+\|_{L^q}$ for every $q >1$.

One argues similarly on $u^-$ and obtains a bound on every $\int_{S^n}|u|^q$ ($q>1$) in terms of $\|f\|_{L^1}$.

\section{Application}

Let $m\geq 1$ be an integer and $a_1,\dots, a_m\in\RR^n$ ($n\geq 2$), $\gamma_1,\dots,\gamma_m\in\RR$ and
$p> 1$. Assume $\sum_{i =1}^m\gamma_i=0$ if $p> n$. We are interested in the following
[p. 280]
problem:
\begin{equation}\label{18}
\begin{array}{rl}
    -\div (|Du|^{p-2} Du) = \sum_{i=1}^m\gamma_i\delta(x-a_i)&\qquad\text{ on }\RR^n \\
    u\to 0 &\qquad\text{ as } |x|\to\infty
\end{array}
\end{equation}
By a solution, we mean a $u\in C^1(\RR^n \setminus\{a_1,\dots, a_m\})$ which satisfies $|Du|^{p-1}\in L^1_{\text{loc}}$
and solves (18) in the weak sense.

We then have (Kichenassamy [5]):
\begin{theorem}
    Problem (18) has a solution. It is unique if one imposes the condition
    \begin{equation}\label{19}
        u-\sum_{i=1}^m\gamma_i\varphi(x-a_i)\in L^\infty(\RR^n)
    \end{equation}
    where $\varphi(x)=C|x|^{(p-n)/(p-1)}$ (resp.\ $C\ln(1/|x|)$ if $p=n$) and $C$ is adjusted to make $\varphi$ a solution of $-\div (|D\varphi|^{p-2} D\varphi)=\delta$.
\end{theorem}

These problems are motivated by models of quark confinement.

We sketch the proof in the case $p = n$. The case $p < n$ will be briefly
alluded to, and the case $p > n$ which is based on a different technique,
is omitted. We shall make use of a $C^{1+\alpha}$ regularity theorem of which a proof, in a very general framework, is given in [5].

\begin{proof}[\textsc{Proof (case $p = n$)}.] We seek a solution of $-\div (|Du|^{p-2} Du)  = 0$ on $\RR^n \setminus \{a_1,\dots,a_n\}$ which behaves near each $a_i$ like $\varphi$ near $0$. More precisely, we proceed in two steps:

(a) \underline{An approximate solution.} We choose a sequence $(f_k)$ of functions supported in a fixed ball, of bounded $L^1$ norm, and that tend weakly to $\sum_{i=1}^m\gamma_i\delta(x-a_i)$. It is easily seen that if we can solve
\begin{equation}\label{20}
\begin{array}{rl}
    Au_k = f_k &\text{in }\RR^n\qquad (A:=-\div (|Du|^{p-2} Du) )\\
    u_k\to 0     &\text{at infinity, }
\end{array}
\end{equation}
with $u_k$ bounded (independently of $k$) on spheres around the $a_i$'s then
[p.~281]
\begin{equation}\label{21}
    \left| u_k-\sum_{i=1}^m\gamma_i\varphi_k(x-a_i)\right|\leq C
\end{equation}
where $\varphi_k$ tends to $\varphi$ in a suitable way. It will follow from this estimate that one can pass to the limit $k\to+\infty$ and obtain a solution $u$ of $Au = 0$ on $\RR^n \setminus \{a_l,\dots,a_m\}$ which, having the "right singularities," is in fact a solution to our original problem. We focus here on the estimates on $u_k$.

(b) \underline{Estimates on $u_k$.} $\pi$ being the stereographic projection,
$u_k\circ\pi$ is a solution of $Aw_k = g_k$ where $Aw$ is $-\div_{S^n}(|Du|^{p-2} Du)$ on $S^n$, and $\la f_k=g_k$.
It follows that $w_k$, and hence $u_u$ is bounded in $L^1_{\text{loc}}$ for $q > n$ say. It further follows by a combination of regularity techniques (see [5]) that $u_k$ is locally bounded in $C^{1+\alpha}$ on every compact set away from $\{a_l,\dots,a_m\}$.

This is enough to pass to the limit uniformly on every
$K \Subset \RR^n\setminus\{a_l,\dots,a_m\}$. Relation (21) gives in the limit
\begin{equation}\label{22}
    \left| u-\sum_{i=1}^m\gamma_i\varphi(x-a_i)\right|\leq C
\end{equation}
from which the proof follows.
\end{proof}
\begin{remk} \emph{When $p < n$ one can work on approximate solutions $u_k$ given by
\begin{equation}\label{23}
Au_k=f_k\text{ in } | x|< k
\end{equation}
with Dirichlet conditions. One can estimate $u_k$ by classical symmetrization
theorems: $|u_k|^*\leq C\varphi$ (because the relevant symmetrized problem has its solution uniformly bounded by $\varphi$, independently of $k$). One then notices that, by the Hardy-Littlewood theorem,
\begin{equation}\label{24}
\int_{|x|\leq k}e^{-|x|^2}|u_k|^q\leq C\int_{\RR^n}e^{-|x|^2}\varphi^q
\end{equation}
with $q > p-1$ chosen so that the right-hand side converges. The $C^{1+\alpha}$ bound on the $u_k$'s now follows from a regularity theorem (see [5]).}
\end{remk}
[p.~282]

\begin{quote}
\end{quote}
Satyanad Kichenassamy

\end{document}